\documentclass[preprint,12pt]{elsarticle}
\usepackage{doi}
\usepackage[british]{babel}
\usepackage{amsthm, amsmath, amssymb, tikz, ifthen}

\definecolor{limoncello}{rgb}{.95, 1, .2}

\usetikzlibrary{positioning}
\usetikzlibrary{arrows}
\usetikzlibrary{fadings}
\usetikzlibrary[topaths]
\usetikzlibrary{calc}

\journal{Discrete Mathematics}
 
\DeclareMathOperator{\ld}{ld}

\newcommand{\ZZ}{\mathbb{Z}}

\newcommand{\card}[1]{
|#1|
}

\newtheorem{theorem}{Theorem}[section]

\newtheorem{corollary}[theorem]{Corollary}
\newtheorem{lemma}[theorem]{Lemma}
\newtheorem{observation}[theorem]{Observation}

\newtheorem{proposition}[theorem]{Proposition}

\definecolor{ZERO}{rgb}{1,0,0}
\definecolor{ONE}{rgb}{0,1,0}
\definecolor{TWO}{rgb}{0,0,1}
\definecolor{grey}{rgb}{.47,.47,.47}
\definecolor{very light grey}{rgb}{.93,.93,.93}
\definecolor{dark grey}{rgb}{.2,.2,.2}

\newcommand{\wlg}{without loss of generality}
\newcommand{\cf}{c.f.}

\newcommand{\dfeq}{:=}

\newcommand{\opair}[2]{( #1, #2)}
\newcommand{\Opair}[2]{\big( #1, #2\big)}
\newcommand{\upair}[2]{\{ #1, #2\}}

\newcounter{tiefe}
\begin{document}

\begin{frontmatter}
 


\title{New bounds on the Ramsey number $r(I_m, L_n)$}


\author{Ferdinand Ihringer\fnref{fundingFI}}
\ead{ferdinand.ihringer@ugent.be}
\address{Ghent University. Department of Mathematics: Analysis, Logic, and Discrete Mathematics.}
\fntext[fundingFI]{The first author acknowledges support from a PIMS Postdoctoral Fellowship and ERC advanced grant 320924. The author is supported by a postdoctoral fellowship of the Research Foundation --- Flanders (FWO).}

\author{Deepak Rajendraprasad}
\ead{deepak@iitpkd.ac.in}
\address{Indian Institute of Technology Palakkad}

\author{Thilo Weinert\fnref{fundingTVW}}
\ead{thilo.weinert@univie.ac.at}
\address{Universit\"at Wien \\
Institut für Mathematik \\
Kurt G\"odel Research Center \\
Augasse 2-6, UZA 1 - Building 2 \\
1090 Wien \\
AUSTRIA}
\fntext[fundingTVW]{The last author acknowledges support by the Israel Science Foundation grant number 1365/14 and the FWF grant number Y1012-N35}

\begin{abstract}
We investigate the Ramsey number $r(I_m, L_n)$ which is the smallest natural number $k$ such that every oriented graph on $k$ vertices contains either an independent set of size $m$ or a transitive tournament on $n$ vertices. Continuing research by Larson and Mitchell and earlier work by Bermond we establish two new upper bounds for $r(I_m, L_3)$ which are paramount in proving $r(I_4, L_3) = 15 < 23 = r(I_5, L_3)$ and $r(I_m, L_3) = \Theta(m^2 / \log m)$, respectively. We furthermore elaborate on implications of the latter on upper bounds for $r(I_m, L_n)$.
\end{abstract}


\end{frontmatter}

\section{Introduction}

In this paper the minimal number $\ell$ for which every oriented graph\footnote{We use the adjective ``oriented'' over ``directed'' as the graphs under discussion contain at most one edge between any two vertices. Likewise, the graphs are all loopless.} of order $\ell$ either contains an independent set of cardinality $m$ or a transitive induced subtournament of order $n$ is studied. This minimal number $\ell$ is denoted by $r(I_m, L_n)$.

The case $m = 2$ received a decent amount of attention, it is known that $r(I_2, L_3) = 4$, \cf\ \cite{956ER0}, that $r(I_2,L_4) = 8$, \cf\ \cite{964EM0} and that $r(I_2, L_5) = 14$ and $r(I_2,L_6) = 28$, \cf\ \cite{970RP0}. The general asymptotic behaviour of $r(I_2, L_n)$ was studied as well, Stearns in \cite{959S0} showed that $r(I_2, L_n) \leqslant 2^{n - 1}$, this was later improved to $r(I_2, L_n) \leqslant 7\cdot 2^{n - 4}$ for $n > 4$ by Reid and Parker in \cite{970RP0} and to $r(I_2, L_n) \leqslant 55 \cdot 2^{n - 7}$ for $n > 6$ by S{\'a}nchez-Flores in \cite{994S1}. Erd\H{o}s and Moser established $r(I_2, L_n) \geqslant 2^{(n - 1)/2}$ in \cite{964EM0}.  This case was furthermore studied in the papers \cite{971M0, 017MS0} and \cite{998S1}.

\begin{figure}
\centering
\begin{tikzpicture}[->,>=stealth',shorten >=13pt, shorten <=13pt,
                    thick,scale=0.9]
\def \n {8}
\def \radius {1cm}
\def \margin {11} 

\pgfmathsetmacro\nn{int(\n - 1)}

\foreach \s in {0,...,\nn}
{
\pgfmathsetmacro\ss{Mod(int(3*\s), 8)}
\pgfmathsetmacro\tt{(40*Mod(int(\s + 1),2)}
  
  \node[draw, circle, fill=very light grey] at ({360/\n * (\s - 1)}:{(Mod(int(\s),2) + 3)*\radius}) {$\pgfmathprintnumber{\ss}$};
  \path ({360/\n * (\s - 1)}:{(Mod(int(\s),2) + 3)*\radius}) edge[->,color=black,thick] 
    ({360/\n * (\s)}:{(Mod(int(\s + 1),2) + 3)*\radius});
  \path ({360/\n * (\s - 2)}:{(Mod(int(\s + 1),2) + 3)*\radius}) edge[<-,color=black,thick, bend right = \tt] 
    ({360/\n * (\s)}:{(Mod(int(\s + 1),2) + 3)*\radius});
}
\end{tikzpicture}
\caption{Bermond's $\upair{I_3}{L_3}$-free graph on $8$ vertices.}
\label{8vertices}
\end{figure}

By contrast, cases in which $m > 2$ were only studied in considerably fewer papers. In \cite{974B1}, Bermond proved $r(I_3,L_3) = 9$ mainly by providing an example establishing the lower bound. The numbers $r(I_m,L_n)$ for $m > 2$ were last revisited two decades ago by Larson and Mitchell \cite{997LM0}. There they proved $r(I_m,L_3)\leqslant m^2$ using a degree argument and showed $r(I_4,L_3) > 13$. This left open three possibilities for the number $r(I_4, L_3)$, the arguably easiest case among the hitherto open ones. 

There seems to be a noticable gap between the knowledge about undirected Ramsey numbers $r(I_m, K_n)$ and that on oriented Ramsey numbers $r(I_m, L_n)$ and hereby we are attempting a step in closing it. The numbers $r(I_m, K_3)$ are known for $1 < m < 10$, they are $3,6,9,14,18,23,28,36$. The last of these values was established in 1982 by Grinstead and Roberts in \cite{982GR0}. More information on small Ramsey numbers can be found in Radziszowski's survey \cite{994R0}. Moreover since Kim in \cite{Kim1995} established a lower bound of appropriate order of magnitude, we know that $r(I_m, K_3) = \Theta(m^2 / \log m)$.

Even though we are later going to establish an asymptotically better third bound, in Section~\ref{sec:upper_bnds} we provide an upper bound of $m^2 - m + 3$ for $r(I_m, L_3)$ which is better than both the aforementioned asymptotically better bound and the Larson-Mitchell-bound for $m \leqslant 2^{508}$. More importantly, it allows for the determination of $r(I_m, L_3)$ for $m \in \{4, 5\}$ by giving the correct values. Subsequently, in Section~\ref{sec:lower_bnds}, we construct oriented graphs witnessing $r(I_4, L_3) > 14$ and $r(I_5, L_3) > 22$. Thereby we prove the following.

\begin{theorem}\label{thm:small_ramsey_intro}
  $r(I_4, L_3) = 15$ and $r(I_5, L_3) = 23$.
\end{theorem}

Since any orientation of an $\{I_m, K_3\}$-free graph is $\{I_m, L_3\}$-free,
$r(I_m, L_3) \geq r(I_m, K_3)$. Moreover, since every orientation of a graph 
which contains a $K_4$ will contain an $L_3$, $r(I_m, L_3) \leq r(I_m, K_4)$.
In Section \ref{sec:asymptotics}, we use a result of Alon \cite{996A0} to show
that $r(I_m, L_3)$ behaves more like $r(I_m, K_3)$. That is we show the following.
\begin{theorem}\label{thm:gen_asymptotic_intro}
  $r(I_m, L_3) = \Theta(m^2 / \log m)$.
\end{theorem}

Then we follow an argument of Ajtai, Koml\'os, and Szemer\'edi to give, for each $n \geqslant 3$, asymptotic upper bounds on $r(I_m, L_n)$ of the same order as the best known upper bounds for $r(I_m, K_n)$.

More concretely, we extend the result above to the following:

\begin{theorem}\label{thm:gen_asymptotic_extn_intro}
For each natural number $m$ and each natural number $n \geqslant 3$, there exists a universal constant $C_n$ such that
$r(I_m, L_n) \leqslant C_n m^{n-1} / (\log m)^{n-2}$.
\end{theorem}

Finally---in an appendix---we provide a formula which gives the best known upper bounds for small values of $m$ and $n$.

The numbers $r(I_m, L_n)$ are of interest also due to their connection to ordinal Ramsey theory, \cf\ \cite[Chapter 7]{977W0} and \cite[Chapter 2]{010HL0}. In particular, \cite[Theorem 25]{956ER0} amounts to the following:

\begin{theorem}[Erd\H{o}s and Rado \cite{956ER0}]
$r(\omega m,n)=\omega r(I_m,L_n)$ for all natural numbers $m$ and $n$.
\end{theorem}

In \cite{967ER0} Erd\H{o}s and Rado showed that for any infinite initial ordinal $\kappa$ and any natural numbers $m$ and $n$ there is a natural number $\ell$ such that $r(\kappa m, n)\leqslant\kappa\ell$. They conjectured that $\ell$ never depends on $\kappa$. In \cite{974B0} Baumgartner settled this conjecture affirmatively.

\begin{theorem}[Baumgartner \cite{974B0}]
$r(\kappa m,n)=\kappa r(I_m,L_n)$ for all infinite initial ordinals $\kappa$.
\end{theorem}

\section{Preliminaries}

Let $v$ be a vertex of an oriented graph $D = \opair{V}{A}$. We denote the \textit{in-neighbourhood} of $v$ by $N^-(v)$ and the \textit{out-neighbourhood}
of $v$ by $N^+(v)$. Formally, we have $N^-(v) = \{ w \in V: \opair{w}{v} \in A \}$ and $N^+(v) = \{ w \in V: \opair{v}{w} \in A \}$.
We denote the vertices non-adjacent to $v$ by $I(v)$, formally we have $I(v) = V \setminus (\{ v \} \cup N^-(v) \cup N^+(v) )$.
We denote $\card{N^-(v)}$ by $d^-(v)$ and $\card{N^+(v)}$ by $d^+(v)$. We call $d^-(v)$ the \emph{in-degree} of $v$ and $d^+(v)$ its \emph{out-degree}. Whenever we refer to the \emph{degree} of $v$ simpliciter, we mean the sum $d^-(v) + d^+(v)$ of its in- and out-degrees. An oriented graph is $n$-regular, whenever $d^-(v) = d^+(v) = n$ for all its vertices $v$.


\begin{lemma}\label{lemma:neighbourhoods_props}
Let $m$ and $n$ both be natural numbers, let $D = \opair{V}{A}$ be an $\upair{I_m}{L_n}$-free oriented graph and let $v \in V$. Then the following holds:
  \begin{enumerate}
   \item\label{neighbourhoods_props_one} The induced subgraphs on $N^-(v)$ and $N^+(v)$ are $\upair{ I_m }{L_{n-1}}$-free.
   \item The induced subgraph on $I(v)$ is $\upair{ I_{m-1} }{L_{n}}$-free.
  \end{enumerate}
\end{lemma}
\begin{proof}
  To show the first assertion suppose towards a contradiction that $N^-(v)$
  contains a set of vertices $T$ such that the induced subgraph on $T$ is the transitive tournament
  of size $n - 1$. Then $\{ v \} \cup T$ is the transitive tournament of size $n$. This contradicts
  that $D$ is $L_n$-free. For the induced subgraph on $N^+(v)$ one may argue analogously.
  
  To show the second assertion suppose towards a contradiction that $I(v)$ contains
  an independent set $I$ of size $m-1$. Then $\{ v \} \cup I$ is an independent set of
  size $m$. This contradicts
  that $D$ is $I_m$-free.
\end{proof}

This has the following consequences for the case $n=3$.

\begin{corollary}\label{lemma:neighbourhoods_independent}
Let $m$ be a natural number, let $D = \opair{V}{A}$ be an $\upair{I_m}{L_3}$-free oriented graph and let $v \in V$. Then $N^-(v)$ and $N^+(v)$ are independent sets.
  Particularly, $d^-(v), d^+(v) \leqslant m-1$.
\end{corollary}

We now provide a recursive upper bound for $r(I_m, L_n)$.

\begin{lemma}\label{lemma:easy_upper_bound_general_case}
We have $r(I_{m + 1}, L_{n + 1}) \leqslant 2r(I_{m + 1}, L_n) + r(I_m, L_{n + 1}) - 1$ for all natural numbers $m$ and $n$.
Furthermore, if an $\upair{I_{m+1}}{I_{n+1}}$-free oriented graph $D = \opair{V}{A}$ has order $2r(I_{m + 1}, L_n) + r(I_m, L_{n + 1}) - 2$,
then all $v \in V$ satisfy
\begin{enumerate}
 \item $d^-(v) = d^+(v) = r(I_{m+1}, L_{n})-1$ and
 \item $\card{I(v)} = r(I_m, L_{n + 1})-1$.
\end{enumerate}
\end{lemma}
\begin{proof}
Let $D$ be an $\upair{I_{m+1}}{L_{n+1}}$-free oriented graph.
Let $v \in D$. By Lemma \ref{lemma:neighbourhoods_props}, $N^-(v)$ and $N^+(v)$ have at most size $r(I_{m+1}, L_n)-1$ each,
and $I(v)$ has at most size $r(I_m, L_{n+1})-1$. Hence,
\begin{align*}
  \card{V} &\leqslant \card{\{ v \}} + \card{N^-(v)} + \card{N^+(v)} + \card{I(v)}\\
  &\leqslant 1 + 2(r(I_{m+1}, L_n)-1) + r(I_m, L_{n+1})-1\\
  &= 2r(I_{m + 1}, L_n) + r(I_m, L_{n + 1}) - 2.
\end{align*}
This implies the assertion.
\end{proof}


The following lemma goes back to Larson and Mitchell, \cf\ \cite{997LM0}. It follows from Corollary \ref{lemma:neighbourhoods_independent}
in connection with Lemma \ref{lemma:easy_upper_bound_general_case}. We will later improve on 
it with Proposition \ref{proposition : improvement}.

\begin{lemma}[Larson and Mitchell]
\label{larson-mitchell-bound}
$r(I_m, L_3) \leqslant m^2$ for all natural numbers $m \geqslant 2$.
\end{lemma}

A proof of the following lemma can be found in \cite{959S0} so we do state, yet not prove it.

\begin{lemma}
\label{lemma : exponential bound}
$r(I_2, L_n) \leqslant 2^{n - 1}$ for all natural numbers $n \geqslant 2$.\qed
\end{lemma}



\section{Improving the Larson-Mitchell Upper Bound}\label{sec:upper_bnds}

In this section we improve Lemma~\ref{larson-mitchell-bound} and show that 
$r(I_m, L_3) \leqslant m^2 - m + 3$ for all $m \geqslant 3$. 
This upper bound turns out to be tight for $m \in \{3,4,5\}$.

If $D = \opair{V}{A}$ is an oriented graph and $B, C \subset V$, 
let $E(B, C)$ denote the set of edges between vertices in $B$ and vertices in $C$, 
irrespective of their direction. Formally we have $E(B, C) = A \cap \left(\left(B \times C\right) \cup \left(C \times B\right)\right)$.

For the following lemma and its proof, note that whenever we refer to a triangle without specifying that it be transitive or cyclic, it may be either. In particular, when we say that a subset $S$ of vertices in $D$ contains a triangle, we mean the oriented subgraph of $D$ induced by $S$ contains either a transitive or a cyclic triangle. 

\begin{lemma}\label{lemma:edges_m_eq_3}
Up to isomorphism there is exactly one $\upair{I_3}{L_3}$-free oriented graph $D$ on eight vertices. It has the following properties:
  \begin{enumerate}
   \item \label{edges:firstitem} $D$ is $2$-regular,
   \item every triple of vertices of $D$ contains at least one edge,
   \item the non-neighbourhood of any vertex of $D$ induces a triangle, 
   \item any set of $5$ vertices in $D$ either contains a triangle or the induced 
	 underlying unoriented graph is isomorphic to $C_5$,
   \item any set of $6$ vertices in $D$ contains a triangle.
  \end{enumerate}
\end{lemma}

\begin{proof}
  As $r(I_2, L_3) = 4$ and $r(I_3, L_2) = 3$, the bound in Lemma \ref{lemma:easy_upper_bound_general_case} is tight.
  Hence, the oriented graph is $2$-regular. The second and third assertion follow from $D$
  being $I_3$-free.
  
  Let $M$ be a set of five vertices of $D$. We assume
  that $M$ does not contain a triangle. We can ignore the orientation of the edges. 
  Let $x \in M$. By part $3$, $\card{I(x) \cap M} \leqslant 2$. 
  If $\card{I(x) \cap M} \leqslant 1$, then $\card{M \cap (N^+(x) \cup N^-(x))} \geqslant 3$, so
  part $2$ implies the assertion. Hence, $\card{I(x) \cap M} = 2$ for all $x \in M$.
  Hence, the induced underlying subgraph on $M$ is isomorphic to a cycle of length $5$.
  This implies the fourth assertion. The fifth assertion follows similarly.

  Now we show the uniqueness of $D$. W.l.o.g. the vertex set of $D$ is 
  $\{ 0, 1, 2, 3, 4, 5, 6, 7 \}$, where $N^+(0) = \{ 2, 3\}$, $N^-(0) = \{ 5, 6 \}$,
  and $I(0) = \{ 1, 4, 7 \}$. As $I(0)$ is $\upair{I_2}{L_3}$-free, w.l.o.g. we have the edges
  \begin{align*}
    \opair{1}{4}, \opair{4}{7}, \text{ and } \opair{7}{1}
  \end{align*}
  in $D$. By definition, we have
  \begin{align*}
    \opair{0}{2}, \opair{0}{3}, \opair{5}{0}, \text{ and } \opair{6}{0}
  \end{align*}
  in $D$. As every vertex in $D$ has degree $4$,
  \begin{align*}
    \card{E(N^+(0) \cup N^-(0), I(0)) } = 4 \card{ I(0) } - 2 \card{ E(I(0), I(0))} = 12 - 6 = 6.
  \end{align*}
  Hence,
  \begin{align*}
    2 \card{ E(N^+(0), N^-(0)) } = 4 \cdot 3 - \card{E(N^+(0) \cup N^-(0), I(0)) } = 6.
  \end{align*}
  As $D$ is $L_3$-free, the three edges in $E(N^+(0), N^-(0))$ go from $N^+(0)$
  to $N^-(0)$, so w.l.o.g. we can assume that the edges
  \begin{align*}
    \opair{3}{6}, \opair{2}{5}, \text{ and } \opair{3}{5}
  \end{align*}
  are in $D$. As $3$ has in-degree $2$, there is one edge from $I(0)$ to $3$,
  w.l.o.g. that is $\opair{1}{3}$. As $I(3)$ is $\upair{I_2}{L_3}$-free, the edges
  $\opair{7}{2}$ and $\opair{2}{4}$ are in $D$. Similarly, $I(2)$ is $\upair{I_2}{L_3}$-free, so $\opair{6}{1}$ is an
  edge of $D$. As the out- and in-degrees of all vertices are $2$, the edges
  $\opair{5}{7}$ and $\opair{4}{6}$ are in $D$. Now we have given all $16$ oriented edges of $D$ without loss of generality.
\end{proof}

The unique $\upair{I_3}{L_3}$-free oriented graph on eight vertices may be defined on $\ZZ_{8}$ by setting both $x\mapsto x+1$ and $x\mapsto x-2$,
see Figure \ref{8vertices}.

\begin{lemma}
\label{lemma : >37}
An $\upair{I_4}{L_3}$-free oriented graph on fourteen vertices contains at least $38$ edges.
\end{lemma}
\begin{proof}
We show the statement by contradiction.
Let $D = \opair{V}{A}$ be a $14$-vertex $\{I_4, L_3\}$-free oriented graph with $|A| < 42$.
By Corollary \ref{lemma:neighbourhoods_independent}, every vertex has in-degree and out-degree at most $3$.
Since $|A| < 42$, the sum of in-degrees is less than $42$ and hence there exists a vertex $v$ with 
in-degree at most $2$. By Lemma \ref{lemma:neighbourhoods_props}, $d^-(v) + d^+(v) \geqslant 14 - r(I_3, L_3) = 5$.
Since $d^+(v) \leqslant 3$, we have $d^-(v) = 2$ and $d^+(v) = 3$.

Since $N^+(v)$ is already an independent set of size $3$, and $D$ is $I_4$-free,
each of the $8$ vertices in $I(v)$ is adjacent to at least one vertex of $N^+(v)$.
Hence $\card{E(N^+(v), I(v))} \geqslant 8$. 

Similarly, we get $\card{E(N^-(v), I(v))} \geqslant 5$: Assume to the contrary that $\card{E(N^-(v), I(v))} < 5$. Let $F$ denote
the vertices in $I(v)$ which are not in an edge of 
$E(N^-(v), I(v))$. As $\card{I(v)} = 8$, we have $\card{F} \geqslant 4$.
As $r(I_2, L_3) = 4$, we find an independent set $F'$ of
size $2$ in $F$ and thus $F' \cup N^-(v)$ is an independent set of size $4$. This contradicts that $D$ is $I_4$-free.

By Lemma \ref{lemma:edges_m_eq_3}, $\card{E(I(v), I(v))} = 16$. Let $y$ be the number of vertices $w \in I(v)$
with $d^-(w)+d^+(w) = 6$. Then
\begin{align*}
  45 &\leqslant \card{E(N^+(v), I(v))} + \card{E(N^-(v), I(v))} + 2  \cdot \card{E(I(v), I(v))} \\&= 6y + 5(8-y).
\end{align*}
Hence, $y \geqslant 5$.
This, together with the handshaking lemma, ensures that
there are at most $8$ vertices $u$ with $d^-(u)+d^+(u) = 5$.
As all the vertices $w$ in $D$ satisfy $d^-(w) + d^+(w) \in \{ 5, 6\}$, 
we have $|A| \geqslant (8 \cdot 5 + 6 \cdot 6)/2 = 38$.
\end{proof}

One can improve the previous argument to show that there are at least $41$ edges, but it is slightly more tedious and not needed in the following.

\begin{lemma}\label{lemma:lower_bound_edges}
 Let $m$ be a natural number and suppose that $D = \opair{V}{A}$ is an $\upair{I_m}{L_3}$-free oriented graph. Let $v \in V$ with $d^-(v) = m-1$ and $w \in V$ with $d^+(w) = m-1$. Then
 \begin{align*}
  &\card{E(N^-(v), I(v))} \geqslant 2(\card{I(v)} - m + 1) \intertext{ and }
  &\card{E(N^+(w), I(w))} \geqslant 2(\card{I(w)} - m + 1).
 \end{align*}
\end{lemma}
\begin{proof}
  By Corollary \ref{lemma:neighbourhoods_independent}, $N^-(v)$ is an independent set of size $m-1$. 
  Notice that each $x \in I(v)$ is adjacent to at least one vertex of $N^-(v)$ as otherwise
  $N^-(v) \cup \{ x \}$ is an independent set of size $m$.
  We call $x \in I(v)$ a \emph{private neighbour} (with respect to $N^-(v)$) if $x$ has
  exactly one neighbour in $N^-(v)$. We claim that a vertex $u \in N^-(v)$ is adjacent
  to at most two private neighbours.
  
  Suppose that $u$ is adjacent to three private neighbours, call them $x, y$ and $z$.
  If $x, y$ and $z$ are all adjacent, then the induced subgraph on $\{ u, x, y, z \}$ is
  an $\upair{I_2}{L_3}$-free graph. This contradicts $r(I_2, L_3) = 4$.
  If \wlg{} $x$ and $y$ are not adjacent, then $\{ x, y \} \cup N^-(v) \setminus \{ u \}$
  is an independent set of size $m$. This contradicts thats $D$ is $I_m$-free.
  This shows our claim.
  
  Hence, each $u \in N^-(v)$ is adjacent to at most two private neighbours in $I(v)$.
Let P denote the set of private neighbours in $I(v)$ (with respect to
$N^-(v)$). Since every vertex in $I(v) \setminus P$ has at least two neighbours
in $N^-(v)$ and $\card{P} \leqslant 2\card{N^-(v)} = 2d^-(v)$, we have: 
  \begin{align*}
    \card{E(N^-(v), I(v))} \geqslant & \card{P} + 2(\card{I(v)} -\card{P}) = 2\card{I(v)} - \card{P} \\
    \geqslant & 2(\card{I(v)} - d^-(v)) = 2(\card{I(v)} - m +1).
  \end{align*}
  An analogous argument shows 
  \begin{align*}
    \card{E(N^+(w), I(w))} &\geqslant 2(\card{I(w)} - m + 1).
  \end{align*}
  The assertion follows.
\end{proof}


\begin{proposition}
\label{proposition : improvement}
If $m$ is a natural number, where $m \geqslant 2$, then (1) every $\upair{I_m}{L_3}$-free oriented graph on $m^2 - m + 2$ vertices has at least $(m^2 - m + 2)(2m - 3) / 2$ edges and that (2) $r(I_m, L_3) \leqslant m^2 - m + 3$.
\end{proposition}
\begin{proof}
The proposition 
will be established by induction on $m$.
As there are no $\upair{I_2}{L_3}$-free oriented graphs on four vertices, the statement of the proposition is vacuously true in the case $m = 2$. By Lemmas \ref{larson-mitchell-bound} and \ref{lemma:edges_m_eq_3}\eqref{edges:firstitem} it holds in case $m = 3$ as well. Henceforth we assume $m \geqslant 3$ and the truth of the proposition for $m$. We show it holds for $m + 1$ as well. To this end, let $D = \opair{V}{A}$ an $\upair{I_{m + 1}}{L_3}$-free graph with $|V| \geqslant (m + 1)^2 - (m + 1) + 2 = m^2 + m + 2$.

Note that this induction is slightly twisted. First we use the induction hypothesis on (2) for $m$ to show (1) for $m+1$. Then we use the induction hypothesis on (1) and (2) for $m$ to show (2) for $m+1$.

In order to show that there are no fewer edges in $D$ than claimed, consider that as $D$ is $\upair{I_{m + 1}}{L_3}$-free, by Lemma \ref{lemma:easy_upper_bound_general_case} the non-neighbourhood $I(v)$ of any vertex $v \in V$ induces an $\upair{I_m}{L_3}$-free oriented subgraph $D_v = \Opair{I(v)}{A \upharpoonright (I(v) \times I(v))}$ of $D$. By our induction hypothesis we have $\card{I(v)} \leqslant m^2 - m + 2$ and hence $d^-(v) + d^+(v) \geqslant (m^2 + m + 2) - 1 - (m^2 - m + 2) = 2(m + 1) - 3$ for all $v \in V$. Hence the number of edges in $D$ is at least
\begin{align*}
\frac{1}{2}\sum_{v \in V} \big(d^-(v) + d^+(v)\big)\geqslant \frac{1}{2}\sum_{v \in V} \big(2(m + 1) - 3\big) =
\frac{\card{V}\big(2(m + 1) - 3\big)}{2}.
\end{align*}
For $|V| = m^2+m+2$, this shows (1) for $m+1$.

It remains to show that $|V| > m^2+m+2$ cannot occur.
Assume towards a contradiction that $|V| = m^2 + m + 3$. 
As $r(I_{m+1}, L_2) = m+1$ and, by induction hypothesis, $r(I_m, L_3) \leqslant m^2 - m + 3$, we have equality in Lemma \ref{lemma:easy_upper_bound_general_case}.
By Lemma \ref{lemma:easy_upper_bound_general_case}, $d^-(v) = d^+(v) = m$ and $\card{I(v)} = m^2-m+2$ for all $v \in V$. 

Let us fix $v$. As $d^-(v) = d^+(v) = m$, we can apply Lemma \ref{lemma:lower_bound_edges}
and obtain
\begin{align}\label{eq:lower_bnd}
  \card{E(N^-(v) \cup N^+(v), I(v))} \geqslant 4(\card{I(v)} - m) = 4(m^2-2m+2).
\end{align}
As $d^-(w) = d^+(w) = m$ for $w \in I(v)$, we have that
\begin{align}\label{eq:upper_bnd}
  \card{E(N^-(v) \cup N^+(v), I(v))} + 2\card{E(I(v), I(v))} \\
  = 2m \cdot \card{I(v)} = 2m(m^2-m+2). \notag
\end{align}
Now we will employ our knowledge about the degrees and the induction hypothesis 
for the number of edges in an $\upair{I_m}{L_3}$-free oriented graph on $m^2 - m + 2$ vertices.
We distinguish three cases (a), (b), and (c). Let $\ell := \card{E(N^-(v) \cup N^+(v), I(v))}$.
Case (a): If $m=3$, then, by Lemma \ref{lemma:edges_m_eq_3}, $\card{E(I(v), I(v))} \geqslant 16$.
	Then, by \eqref{eq:upper_bnd}, $\ell \leqslant 16$.
	But, by \eqref{eq:lower_bnd}, $\ell \geqslant 20$. 
	This is a contradiction.
Case (b): If $m=4$, then by Lemma \ref{lemma : >37}, $\card{E(I(v), I(v))} \geqslant 38$.
	Then, by \eqref{eq:upper_bnd}, $\ell \leqslant 36$.
	But, by \eqref{eq:lower_bnd}, $\ell \geqslant 40$. 
	Again, this is a contradiction.
Case (c): If $m > 4$, then we have $\card{E(I(v), I(v))} \geqslant (2m - 3)(m^2 - m + 2) / 2$
    by the induction hypothesis. By 
    \eqref{eq:upper_bnd},
    \begin{align*}
      \ell = \card{E(N^-(v) \cup N^+(v), I(v))} \leqslant 3m^2 - 3m + 6.
    \end{align*}
    As $m > 4$, this contradicts 
    \eqref{eq:lower_bnd}.
\end{proof}

\section{Constructive Lower Bounds}\label{sec:lower_bnds}

\begin{figure}
\centering

\begin{tikzpicture}[->,>=stealth',shorten >=13pt, shorten <=13pt,
                    thick,rotate=14,scale=1.3]
\def \n {7}
\def \radius {1.45cm}
\def \margin {11} 

\foreach \s in {1,...,\n}
{
  \pgfmathsetmacro\ss{Mod(int(7-6*\s), 14)}
  \pgfmathsetmacro\tt{Mod(int(-6*\s), 14)}
  
  \node[draw, circle, thick, scale=0.75, minimum size=30dd, fill=very light grey] at ({360/\n * (\s - 1)}:\radius) {$\pgfmathprintnumber{\tt}$};
  \path ({360/\n * (\s - 2)}:\radius) edge[<-,color=black,thick,bend left=8] 
    ({360/\n * (\s)}:\radius);
  \path ({360/\n * (\s - 4)}:\radius) edge[->,color=black,thick,bend right=5] 
    ({360/\n * (\s)}:\radius);
  
  \node[draw, circle, thick, scale=0.75, minimum size=30dd, fill=very light grey] at ({360/\n * (\s - 1)}:{2*\radius}) {$\pgfmathprintnumber{\ss}$};
  \path ({360/\n * (\s - 1)}:{2*\radius}) edge[->,color=black,thick,bend right=5] 
    ({360/\n * (\s)}:{2*\radius});
  \path ({360/\n * (\s - 2)}:{2*\radius}) edge[<-,color=black,thick,bend right=15] 
    ({360/\n * (\s)}:{2*\radius});
    
  \path ({360/\n * (\s - 1)}:{2*\radius}) edge[->,color=black,thick,bend right=5] 
    ({360/\n * (\s)}:{\radius});
  \path ({360/\n * (\s - 1)}:{\radius}) edge[->,color=black,thick,bend right=5] 
    ({360/\n * (\s)}:{2*\radius});
}
\end{tikzpicture}
\caption{An oriented graph showing $r(I_4,L_3)>14$.}
\label{fourteen}
\end{figure}

\begin{figure}
\centering

\begin{tikzpicture}[->,>=stealth',shorten >=13pt, shorten <=13pt,
                    thick,rotate=14,scale=1.4]
\def \n {22}
\def \radius {3.3cm}
\def \margin {11} 

\foreach \s in {1,...,\n}
{
  \pgfmathsetmacro\ss{int(\s - 1)}
  
  \node[draw, circle, scale=0.8, thick, fill=very light grey, minimum size=25dd] at ({360/\n * (\s - 1)}:\radius) {$\pgfmathprintnumber{\ss}$};
  \path ({360/\n * (\s - 1)}:\radius) edge[->,color=black,bend right=0] 
    ({360/\n * (\s)}:\radius);
  \path ({360/\n * (\s - 4)}:\radius) edge[->,color=dark grey,bend right=70, looseness=1.5] 
    ({360/\n * (\s)}:\radius);
  \path ({360/\n * (\s - 10)}:\radius) edge[->,color=grey,bend right=0] 
    ({360/\n * (\s)}:\radius);
  \path ({360/\n * (\s - 5)}:\radius) edge[<-,color=black,bend left=20] 
    ({360/\n * (\s)}:\radius);
}
\end{tikzpicture}
\caption{An oriented graph showing $r(I_5,L_3)>22$.}
\label{twentytwo}
\end{figure}

\begin{observation}\label{obs:ramsey_I4_L3}
$r(I_4, L_3) = 15$.
\end{observation}

\begin{proof}
By Proposition \ref{proposition : improvement}, $r(I_4, L_3) \leqslant 15$.
The oriented $\upair{I_4}{L_3}$-free graph in Figure \ref{fourteen} may be defined on $\ZZ_{14}$ by setting both $x\mapsto x+1$ and $x\mapsto x-2$ for all $x\in \ZZ_{14}$ and moreover $x\mapsto x+4$ if $x$ is even and $x\mapsto x-6$ if $x$ is odd.
\end{proof}

We want to remark that there is no oriented $\upair{I_4}{L_3}$-free Cayley graph on $14$ vertices.

\begin{observation}\label{obs:ramsey_I5_L3}
$r(I_5, L_3) = 23$.
\end{observation}

\begin{proof}
By Proposition \ref{proposition : improvement}, $r(I_5, L_3) \leqslant 23$.
The oriented $\upair{I_5}{L_3}$-free graph in Figure \ref{twentytwo} may be defined on $\ZZ_{22}$ by setting both $x\mapsto x+1$, $x\mapsto x+4$, $x\mapsto x-5$ and $x\mapsto x+10$ for all $x\in \ZZ_{22}$.
\end{proof}

Both observations together imply Theorem \ref{thm:small_ramsey_intro}.

\section{Probabilistic Upper Bounds}\label{sec:asymptotics}

In this section, we use a result of Alon to show that 
$r(I_m, L_3)$ is in $O(m^2 / \log m)$. 
This bound is better than the one in Proposition \ref{proposition : improvement} 
for large enough $m$. Moreover, this is tight upto multiplicative constants 
since $r(I_m, L_3) \geqslant r(I_m, K_3)$.
Then we follow an upper bound argument of Ajtai, Koml\'os and Szemer\'edi for
$r(I_m, K_n)$ to obtain upper bounds of commensurate order for $r(I_m, L_n)$.

Note that $\ld$ stands for \emph{logarithm dualis}, the logarithm to base $2$.

\begin{proposition}[{\cite[Prop. 2.1]{996A0}}]\label{prop:alon_bound}
Let $G = \opair{V}{E}$ be a graph on $v$ vertices with maximum degree $d \geqslant 1$, in which the neighbourhood of any vertex is $r$-colourable. Then
\begin{align*}
\alpha(G)\geqslant\frac{v\ld d}{160d\ld(r+1)}.
\end{align*}
\end{proposition}

\begin{corollary}\label{cor:tight_asymptotic_bnd}
\label{upper bound}
\begin{align*}
r(I_m, L_3)\leqslant\frac{508 m^2}{\ld m} \text{ for all natural numbers } m \geqslant 2.
\end{align*}
\end{corollary}
\begin{proof}
Assume towards a contradiction that there are a natural number $m$ and an oriented graph $D$ on $v\dfeq 508 m^2/\ld m$ vertices with no transitive triangle and no independent set of size $m$. Let $G$ be the undirected graph obtained from $D$ by forgetting the directions of the edges. Let $d$ denote the maximal degree of a vertex in $G$. We distinguish two overlapping cases:

Firstly we assume $d \leqslant 253$. Then, by Tur\'an's bound, we have
\begin{align*}
m > \frac{v}{d + 1} = \frac{508 m^2}{(d + 1)\ld m} \geqslant \frac{508 m^2}{(253 + 1)\ld m} = \frac{508 m^2}{254 \ld m} = \frac{2m^2}{\ld m}.
\end{align*}
This clearly implies $\ld m > 2m$, so $m > 2^{2m}$, a contradiction.

Secondly we assume $d \geqslant 3$. Since $d < 2m$ by Corollary \ref{lemma:neighbourhoods_independent} and $\ld(x)/x$ is decreasing for $x > e$,
\begin{align*}
\frac{\ld d}{d} \geqslant \frac{\ld(2m)}{2m},
\end{align*}
Note that the neighbourhood of any vertex  $x$ is $2$-colourable since it consists of the in- and out-neighbourhoods of $x$, both of which are independent sets so
by Proposition \ref{prop:alon_bound}, we may conclude that
\begin{align*}
m > \alpha(G)\geqslant & \frac{508 m^2}{\ld m}\cdot\frac{\ld(d)}{160\cdot d\ld3}\geqslant \frac{508 m^2}{\ld m}\cdot\frac{\ld(2m)}{160\cdot2m\ld3}\\
= & \frac{508m(\ld2+\ld m)}{320\ld3\ld m}=\frac{127 m}{80\ld3}(1+\frac{1}{\ld m}).
\end{align*}
It follows that $127<80\ld 3 < 126.8$ which is a contradiction.
\end{proof}

Due to Kim \cite{Kim1995}, $r(I_m, K_3) \geqslant \Theta(m^2 / \log m)$. 
Since $r(I_m, L_3) \geqslant r(I_m, K_3)$, this shows Theorem \ref{thm:gen_asymptotic_intro}.

We follow an argument by Ajtai, Koml\'os, and Szemer\'edi
for $\upair{I_m}{K_n}$-free graphs \cite{Ajtai1980} to obtain another upper bound for $\upair{I_m}{L_n}$-free graphs.

Following the proof of \cite[Lemma 4]{Ajtai1980}, which is a standard application of Chebyshev's and
Markov's inequalities, we obtain the following lemma.

\begin{lemma} \label{lemma:aks_lemma}
 Let $D$ be an oriented graph with $v$ vertices, $e$ edges, $h$ transitive triangles, and average degree $d$.
 Let $0 < p < 1$.
 Then there exists an induced subgraph $D'$ of $D$ with $v'$ vertices, $e'$ edges, $h'$ transitive triangles, and average degree $d'$
 satisfying
 \begin{align*}
    v' \geqslant vp/2, && e' \leqslant 3ep^2 && h' \leqslant 3hp^3, && d' \leqslant 6dp.
 \end{align*}
\end{lemma}

We also need the average version of Alon's bound. 
The constant in the bound can be easily verified from the proof there.
\begin{theorem}[{\cite[Theorem 1.1]{996A0}}]\label{thm:alon_bound2}
Let $G = \opair{V}{E}$ be a graph on $v$ vertices with average degree $d \geqslant 1$, in which the neighbourhood of any vertex is $r$-colourable. Then
\begin{align*}
\alpha(G)\geqslant\frac{v\ld(2d)}{640 d\ld(r+1)}.
\end{align*}
\end{theorem}

\begin{lemma}\label{lemma:prob_lower_bnd_independence}
  Let $\varepsilon \leqslant 1$ be positive. If $D$ is an oriented graph with $v$ vertices, average degree $d \geq 1$, and $h \leqslant v d^{2-\varepsilon}$ transitive triangles, then
  \begin{align*}
    \alpha(D) \geqslant \frac{\varepsilon v\ld d}{2^{15}d}.
  \end{align*}
\end{lemma}
\begin{proof}
\begin{align*}
\text{Let } p = \frac{\sqrt{15} - 3}{6d^{1 - \varepsilon / 2}}.
\end{align*}
  By Lemma \ref{lemma:aks_lemma}, we obtain an oriented graph $D'$ with $v'\geqslant vp/2$ vertices,
		average degree $d' \leqslant 6dp = (\sqrt{15} - 3)d^{\varepsilon/2}$, and $h' \leqslant 3hp^3$ transitive triangles. Since $d^{2-\varepsilon} p^2 = (\sqrt{15} - 3)^2/6^2 = (4 - \sqrt{15})/6$, we get
 \begin{align*}
		 h' \leqslant 3h p^3 \leqslant 3vd^{2 - \varepsilon} p^3 \leqslant \left(4 - \sqrt{15}\right)\frac{vp}{2} \leqslant (4 - \sqrt{15})v'.
\end{align*}
  Deleting one vertex from each of the transitive triangles in $D'$ gives us an $L_3$-free oriented graph 
  $D''$ on $v'' \geqslant (\sqrt{15} - 3)v' \geqslant (\sqrt{15} - 3)vp/2$ vertices. So the neighbourhood of any vertex in $D''$ is $2$-colourable. If $d''$ denotes the average degree of $D''$, then
  \begin{align*}\
  d'' \leqslant \frac{d'}{\sqrt{15} - 3} \leqslant \frac{6dp}{\sqrt{15} - 3} = d^{\varepsilon / 2}.
  \end{align*}
  We distinguish two cases:
  
  First we assume that $d'' \leqslant 1959$. Then by Caro-Wei, \cf\ \cite{979C1, 981W1}, we get
  \begin{align*}
   \alpha(D'') \geqslant & \frac{v''}{d''+ 1} \geqslant \frac{v''}{1960} \geqslant \frac{(\sqrt{15} - 3)vp}{3920} = \frac{(\sqrt{15} - 3)^2v}{23520d^{1 - \varepsilon / 2 }} = \frac{6(4 - \sqrt{15})v}{23520d^{1 - \varepsilon / 2}} \\
   = & \frac{(4 - \sqrt{15})v}{3920d^{1 - \varepsilon / 2}} \geqslant \frac{v}{30864d^{1 - \varepsilon / 2}} = \frac{v\sqrt{d^\varepsilon}}{2^4 \cdot 1929d} \geqslant \frac{v\ld d^\varepsilon}{2^{15}d} = \frac{\varepsilon v\ld d}{2^{15}d}.
     \end{align*}
     The last inequality follows from $\ld(x) / \sqrt{x}$ having a global maximum of value smaller than $2/(e \ln 2)$, which is less than $2^{11}/1929$.
     
     Now we assume that $d'' \geqslant e$. Then, as the function $\ld(x) / x$ is decreasing above $e$ and by Theorem \ref{thm:alon_bound2}, 
  \begin{align*}
    \alpha(D) &\geqslant \alpha(D'')
    \geqslant \frac{1}{640 \ld 3} v'' \frac{\ld d''}{d''}
    \geqslant \frac{1}{640 \ld 3} \frac{vp(\sqrt{15} - 3)}{2} \frac{\ld (d^{\varepsilon/2})}{d^{\varepsilon/2}}\\
    &\geqslant \frac{1}{640 \ld 3} \frac{v(\sqrt{15} - 3)}{2} \frac{(\sqrt{15}-3)}{6d} \ld (d^{\varepsilon/2})\\
    &\geqslant \frac{v(4 - \sqrt{15})\varepsilon \ld d}{2^9 5d \ld 3}
    \geqslant \frac{v\varepsilon \ld d}{2^9 5d(4 + \sqrt{15})\ld 3}
    \geqslant \frac{\varepsilon v \ld d}{2^{15}d}.
  \end{align*}
\end{proof}

\begin{theorem}\label{thm:gen_asymptotic}
  For all natural numbers $m, n \geqslant 2$, 
  \begin{align*}
    r(I_m, L_n) \leqslant 2^{17n} \cdot \frac{m^{n-1}}{(\ld m)^{n-2}}.
  \end{align*}
\end{theorem}
\begin{proof}
  We prove 
   the bound by induction on $n$.
  We already know that $r(I_m, L_2) \leqslant m$ and, by Corollary \ref{cor:tight_asymptotic_bnd},
  $r(I_m, L_3) \leqslant 2^9 \cdot \frac{m^2}{\ld m}$. Fix $n \geqslant 4$ and assume that the claim
  is true for $n-1$ and $n-2$.
  
  Suppose that $D$ is an $L_n$-free oriented graph on 
  \begin{align*}
    v \geqslant 2^{17n} \cdot \frac{m^{n-1}}{(\ld m)^{n-2}}
  \end{align*}
  vertices. We will argue that $\alpha(D) \geqslant m$. Let $\varepsilon = \frac{7}{8n - 8}$. Furthermore, let $d$ and $\bar{d}$ denote the maximum and average degrees of the vertices in $D$, respectively.
  
  \paragraph*{Case 1.} $\bar{d} \leqslant 7$. Then by Tur\'an's bound,
  \begin{align*}
  \alpha(D) \geqslant \frac{2^{17n} m^{n - 1}}{(7 + 1)(\ld m)^{n-2}} \geqslant \frac{2^{17} m^{n - 1}}{8(\ld m)^{n-2}} \geqslant \frac{2^{14} m^{n - 1}}{(\ld m)^{n-2}} \geqslant 2^{14} m \geqslant m.
  \end{align*}
  
  \paragraph*{Case 2.} $\bar{d} \geqslant 3$ and the number of transitive triangles in $D$ is at least $v \cdot d^{2-\varepsilon}$.
  The graph $D$ contains at most $vd/2$ edges. By double counting there exists an oriented edge $e = \opair{a}{b}$ in $D$
  such that $\opair{a}{b}$ lies in at least
  \begin{align*}
    \frac{vd^{2-\varepsilon}}{vd/2}
  \end{align*}
  transitive triangles of the form $\{ \opair{a}{b}, \opair{b}{v}, \opair{a}{v} \}$.
  Let $V_e$ denote the set of vertices $v$ such that $\{ \opair{a}{b}, \opair{b}{v}, \opair{a}{v} \}$
  is a transitive triangle of $D$. Then $\card{V_e} \geqslant 2d^{1-\varepsilon}$.
  
  If there is an oriented subgraph $H$ isomorphic to $L_{n-2}$ 
  in the subgraph $D'$ induced on $V_e$, then the induced subgraph on $\{ a, b \} \cup V(H)$
  is isomorphic to $L_n$. Hence, $D'$ is $\upair{I_m}{L_{n-2}}$-free. Hence,
  \begin{align*}
   2d^{1-\varepsilon} \leqslant \card{V_e} < r(I_m, L_{n-2}).
  \end{align*}
  Hence,
  \begin{align*}
    d 	& < \left(\frac{r(I_m, L_{n - 2})}{2}\right)^\frac{1}{1-\varepsilon}
		 < \left(2^{17n - 35} \cdot \frac{m^{n - 3}}{(\ld m)^{n - 4}}\right)^{1 + \frac{7}{8n - 15}} \\
	  	& = 2^{17n - 35 + \frac{119n - 245}{8n-15}} \cdot 
	  	  \frac{m^{n - 3 + \frac{7n-21}{8n-15}}}
		       {(\ld m)^{n - 4 + \frac{7n -28}{8n-15}}} 
	  	 < 2^{17n - 35 + 15} \cdot 
	  	  \frac{m^{n - 3 + \frac{7}{8}}}
		       {(\ld m)^{n - 4}} \\
	  	& = 2^{17n - 20} \cdot 
	  	  \frac{m^{n - \frac{17}{8}}}
		       {(\ld m)^{n - 4}}.
  \end{align*}
  By Tur\'an's bound, 
  \begin{align*}
    \alpha(D) \geqslant \frac{v}{d+1} \geqslant \frac{v}{2d} \geqslant \frac{2^{19} m^\frac{9}{8}}{(\ld m)^2} \geqslant 2^{12}m\geqslant m,
  \end{align*}
  as $\sqrt[8]{m}(\ld m)^{-2}$ has a minimum of $(e\log 2/16)^2 > 2^{-7}$ at $e^{16}$.

  \paragraph*{Case 3.} $\bar{d} \geqslant 3$ and there are fewer than $v \cdot d^{2-\varepsilon}$ transitive triangles in $D$.
  By Lemma \ref{lemma:neighbourhoods_props}\eqref{neighbourhoods_props_one}, we have
  \begin{align*}
    d < 2r(I_m, L_{n-1}) \leqslant 2^{17n - 16} \cdot \frac{m^{n-2}}{(\ld m)^{n-3}}.
  \end{align*}
As $\ld(x)/x$ is decreasing for $x \geqslant 3$, by Lemma \ref{lemma:prob_lower_bnd_independence},
  \begin{align*}
    \alpha(D) \geqslant & \frac{\varepsilon v \cdot \ld \bar{d}}{2^{15}\bar{d} }\geqslant \frac{\varepsilon v\ld d}{2^{15}d} \geqslant \frac{2^{17n - 15}\varepsilon m^{n - 1}\ld d}{d(\ld m)^{n - 2}} \\
        > & \frac{2\varepsilon m(17n - 16 + (n - 2)\ld m - (n - 3)\ld\ld m)}{\ld m} \\
        = & 7m\cdot\frac{17n - 16 + (n - 2)\ld m - (n - 3)\ld\ld m}{4(n - 1)\ld m} \\
        = & \frac{7m}{4}\left(\frac{17}{\ld m} + \frac{1}{(n - 1)\ld m} + \left(1 - \frac{2}{n - 1} \right)\left(1 - \frac{\ld\ld m}{\ld m}\right) + \frac{1}{n - 1}\right)\\
        \geqslant & \frac{7m}{4}\left(\frac{17}{\ld m} + \left(1 - \frac{2}{n - 1} \right)\left(1 - \frac{\ld\ld m}{\ld m}\right) + \frac{1}{n - 1}\right)
        \geqslant m.
     \end{align*}
   To see that the last inequality is true, we distinguish two subcases. First we assume that $(\ld m)^4 \geqslant m$. Then $\ld m \leqslant 16$, so:
   \begin{align*}
      &  \frac{7m}{4}\left(\frac{17}{\ld m} + \left(1 - \frac{2}{n - 1} \right)\left(1 - \frac{\ld\ld m}{\ld m}\right) + \frac{1}{n - 1}\right)\\
   \geqslant &  \frac{7m}{4}\left(\frac{17}{\ld m} + \left(1 - \frac{2}{n - 1} \right)\left(1 - \frac{\ld\ld m}{\ld m}\right)\right)\\
   \geqslant &  \frac{7m}{4}\cdot\frac{17}{\ld m}
    \geqslant \frac{7m}{4}\cdot\frac{17}{16}
    \geqslant \frac{119m}{64}
    \geqslant \frac{9m}{5}
    \geqslant m.
   \end{align*}
 Now we assume that $(\ld m)^4 \leqslant m$. Also recall that $n \geqslant 4$. Then
   \begin{align*}
   &  \frac{7m}{4}\left(\frac{17}{\ld m} + \left(1 - \frac{2}{n - 1} \right)\left(1 - \frac{\ld\ld m}{\ld m}\right) + \frac{1}{n - 1}\right)\\
       \geqslant & \frac{7m}{4}\left(\left(1 - \frac{2}{n - 1} \right)\left(1 - \frac{\ld\ld m}{\ld m}\right) + \frac{1}{n - 1}\right)\\
         \geqslant & \frac{7m}{4}\left(\left(1 - \frac{2}{n - 1} \right)\left(1 - \frac{1}{4}\right) + \frac{1}{n - 1}\right)\\
         \geqslant & \frac{7m}{4}\left(\left(1 - \frac{2}{n - 1} \right)\frac{3}{4} + \frac{1}{n - 1}\right)
         \geqslant \frac{7m}{4}\left(\frac{3}{4} - \frac{1}{2n - 2}\right)\\
         \geqslant &\frac{7m}{4}\left(\frac{3}{4} - \frac{1}{6}\right)
         \geqslant \frac{7m}{4}\cdot\frac{7}{12}
         \geqslant \frac{49m}{48}     
      \geqslant m.
   \end{align*}
 \end{proof}

This implies Theorem \ref{thm:gen_asymptotic_extn_intro}.

\section{Coda}
There are more open problems in finite combinatorics stemming from set theory. Determining $r(I_3, L_4)$ would continue our work and seems feasible given the size of the candidates for examples of $\upair{I_3}{L_4}$-free graphs.


Finally, for the Ramsey numbers $r(\omega^m, n)$ formulae have been found for all natural numbers $m \ne 4$ and all natural numbers $n$ by Nosal in \cite{975N0, 979N0}. The determination of the numbers $r(\omega^4, n)$ by a formula, however, has still to be accomplished.


\subsection*{A Formula for Small $m$ and $n$}

We provide the following---admittedly slightly baroque---formula. It gives asymptotically suboptimal upper bounds for $r(I_m, L_n)$ but provides the state of the art for small $m$ and $n$.
Let 
\begin{align*}
v(m, n) \dfeq \sum_{i = 0}^{n - 2}\binom{i + m - 1}{i+1} 2^i  - \binom{m + n - 6}{m - 4}2^{n - 3} + 1.
\end{align*}

The following proposition can be proved from 
Lemma~\ref{lemma:easy_upper_bound_general_case} by induction on $m$ and $n$
\begin{proposition}
\label{theorem : over the top formula}
We have $r(I_m, L_n) \leqslant v(m, n)$ for all natural numbers $m$ and $n$ with $m \geqslant 2$ and $n \geqslant 3$.
\end{proposition}


\end{document}